\documentclass[11pt]{amsart} \usepackage{amssymb,amsmath,amsthm,mathrsfs}

\usepackage{a4wide}

\usepackage{enumerate,paralist}

\numberwithin{equation}{section} \newtheorem{theorem}{Theorem}[section]    \theoremstyle{definition}
 \newtheorem{remark}[theorem]{Remark} 

\def\q{\mathfrak{q}}

\newtheorem{conj}[theorem]{Conjecture}

\begin{document}

\title[Reidemeister torsion, Complex volume, and Zograf infinite product] 
{Reidemeister torsion, Complex volume, and Zograf infinite product for hyperbolic 3-manifolds with cusps}

\author{Jinsung Park} \address{School of Mathematics\\ Korea Institute for Advanced Study\\ 207-43\\ Hoegiro 85\\ Dong\-daemun-gu\\ Seoul 130-722\\ Korea } \email{jinsung@kias.re.kr}

\thanks{2010 Mathematics Subject Classification 	57Q10, 32Q45, 58J28.}

\date{\today}

\begin{abstract} In this paper, we prove an equality which involves Reidemeister torsion, complex volume, and Zograf infinite product for hyperbolic 3-manifolds with cusps.
\end{abstract}

\maketitle


 \section{Introduction} \label{s:Introduction}

In this paper, we prove an equality which involves Reidemeister torsion, complex volume, and Zograf infinite product for hyperbolic 3-manifolds with cusps.
This partially extends the result for closed hyperbolic 3-manifolds in the previous work  \cite{Park17}. In this equality,  new contributions from cusps are given by Dedekind eta functions and theta functions.

To state the main result of this paper, let us introduce some notations. 
Let $M_0$ denote a hyperbolic 3-manifold with cusps. Then we have a complex valued invariant called \emph{complex volume}
\begin{equation}
\mathbb{V}(M_0)=\mathrm{Vol}(M_0)+i2\pi^2 \mathrm{CS}(M_0)
\end{equation}
where the real part $\mathrm{Vol}(M_0)$ denotes the hyperbolic volume of $M_0$ and the imaginary part $\mathrm{CS}(M_0)$
denotes the Chern-Simons invariant defined by the Levi-Civita connection
of the hyperbolic metric of $M_0$. This complex volume plays a very important role in the research of hyperbolic 3-manifolds of finite volume,
and has been studied extensively in \cite{NZ}, \cite{Y}, \cite{NY}, \cite{Zi}. This complex volume is one of terms in the main result of this paper. 

Another main object appearing in our result is the Reidemeister torsion attached to a certain representation of $\pi_1(M_0)$. This representation is defined 
to be the composition of the $k$-th symmetric tensor of the natural action of $\mathrm{SL}(2,\mathbb{C})$ on $\mathbb{C}^2$ and a $\mathrm{SL}(2,\mathbb{C})$-lift of the 
holonomy representation $\rho:\pi_1(M_0)\to \mathrm{PSL}(2,\mathbb{C})$. We denote by $\rho^k$ the resulting representation of $\pi_1(M_0)$.  
The choice of the $\mathrm{SL}(2,\mathbb{C})$-lifting corresponds to a spin structure on $M_0$. But, when $k=2n$ is even, 
the resulting representation $\rho^{2n}$ does not depend on this choice. Since the definition of Reidemeister torsion also involves the choice of basis of
the homology groups $H_*(M_0, \rho^{2n})$, it is necessary to specify a basis in order to define Reidemeister torsion. 
For a hyperbolic 3-manifold with cusps $M_0$, there is a canonical way to get 
a basis $H_*(M_0,\rho^{2n})$ from a simple closed curve $c_i$ in the torus section $T_i$ associated to the $i$-th cusp of $M_0$ for $i=1,\ldots,h$ (see the subsection \ref{ss:spin}).  Here $h$ denotes the number of cusps.
Let us denote by $\mathcal{T}(M_0, \rho^{2n}, \{c_i\})$
the resulting Reidemeister torsion.

The third object appearing in our result is \emph{Zograf infinite product}, which was introduced in \cite{Z}.  This is defined by
\begin{equation}
 F_{n}(M_0)= \prod_{[\gamma]} \prod_{m=n}^\infty
(1 - \q_{\gamma}^{m}) \qquad \text{for} \quad n\geq 3. 
\end{equation} 
Here the first product is taken over the set of conjugacy classes of the primitive loxodromic elements $\gamma\in \Gamma\subset \mathrm{PSL}(2,\mathbb{C})$
where $\Gamma$ is the image of the holonomy representation $\rho$ of $\pi_1(M_0)$, and
$\q_{\gamma}=\exp(-(l_\gamma+i\theta_\gamma))$ with $l_\gamma$ and $\theta_\gamma$ denoting the length and torsion of the prime geodesic in $M_0$ determined by $[\gamma]$.

In our result for hyperbolic 3-manifolds with cusps, there are contributions from cusps so that we need to introduce some notations for these.
Let us choose a pair of simple closed curves $(m_i,l_i)$ on each torus section $T_i$ associated to the $i$-th cusp, which
form a basis of $H_1(T_i, \mathbb{Z})$. The holonomy representation $\rho$ induces a representation of the  subgroup of $\pi_1(M_0)$ generated by
$(m_i,l_i)$, which determines a complex numbers $\tau_i$ with $\mathrm{Im}(\tau_i)>0$ for $i=1,\ldots, h$. Then $\tau_i$ is the modulus of the Euclidean structure on $T_i$ with respect to $(m_i,l_i)$ 
(see the subsection \ref{ss:deformation}). 

Now we can state  the main result of this paper:

\begin{theorem}\label{t:main theorem} Let $M_0$ be a complete hyperbolic $3$-manifold of finite volume with $h$ cusps. For $n\geq 3$, the following equality holds
\begin{equation}\label{e:main thm} 
\left| \mathcal{T}(M_0,\rho^{2(n-1)}, \{m_i\}) \, \prod_{i=1}^h  \eta({\tau}_i)^{2}\right|^{-1} 
=\left| \exp\left( \frac{1}{\pi} (n^2-n+\frac16) \mathbb{V}(M_0) \right)\, F_{n}(M_0)\right| 
\end{equation}
where $\eta(\tau_i)$ denotes the Dedekind eta function of $\tau_i$.
\end{theorem}

Let us remark that the left hand side of the equality \eqref{e:main thm} does not depend on the choice of a basis $(m_i, l_i)$ of $H_1(T_i,\mathbb{Z})$ since a change of this will also cause a change of $\tau_i$. 
As we can see from its proof in the section \ref{s:pf-main-thm}, these changes cancel each other.
A corresponding equality to \eqref{e:main thm} for $\rho^{2n-1}$ is also given in Theorem \ref{t:theorem-even} where the contributions from cusps are given in terms of the Dedekind eta function and a theta function. 
These equalities can be considered as partial generalizations of the results in \cite{Park17} 
since the equalities in \cite{Park17} hold between complex valued invariants without
modulus sign. Unfortunately the proof used in this paper does not work to handle some terms of modulus 1 appearing in the equalities proved in \cite{Park17}.
In this paper, we present a self-contained proof which does not depend on the main theorems of the previous work \cite{Park17}.
Actually its proof is simpler since we need not deal with more subtle terms of modulus 1 appearing in the results of \cite{Park17}.

For a hyperbolic 3-manifold with cusps $M_0$, by the fundamental work of Thurston \cite{T}, there exists a deformation space $\mathscr{D}(M_0)$
of (in)complete hyperbolic structures on the underlying topological manifold of $M_0$. Let us denote by $M_u$ the corresponding (in)complete hyperbolic 3-manifold for each point $u\in \mathscr{D}(M_0)$. 
There is a corresponding holonomy representation $\rho_u:\pi_1(M_u)\to\mathrm{PSL}(2,\mathbb{C})$, for which one can define a representation $\rho^k_u$ of $\pi_1(M_u)$ as before.
From definitions of Reidemeister torsion and Zograf infinite product, which depend on a hyperbolic structure through $\rho_u$, one can see that these invariants extend to be holomorphic functions
over $\mathscr{D}(M_0)$.  See \eqref{e:Def-Zog} for the definition of the Zograf infinite product for $M_u$. 
By \cite{Y}, the complex volume $\mathbb{V}(M_0)$ also extends to be holomorphic function over an open neighborhood of the origin in $\mathscr{D}(M_0)$. 
These facts and Theorem \ref{t:main theorem} lead the author to make the following conjecture:

\begin{conj}\label{c:odd}
 There exists an open neighborhood $V$ of the origin in $\mathscr{D}(M_0)$ where the following equality holds for $n\geq 3$,
\begin{equation*}\label{e:conjecture} 
  \mathcal{T}(M_u,\rho_u^{2(n-1)}, \{m_i\})^{-12} \, \prod_{i=1}^h  \eta({\tau_i(u)})^{-24} 
=c_{M_0,n} \exp\left( \frac{2}{\pi} (6n^2-6n+1) \mathbb{V}(M_u) \right)\, F_{n}(M_u)^{12} 
 \end{equation*}
where $M_u$ denotes the (in)complete hyperbolic 3-manifold corresponding to $u\in V$ and $c_{M_0,n}$ is a constant depending only on $M_0$ and $n$ with $|c_{M_0,n}|=1$. 
 \end{conj}   

Let us remark that we need to take 12-th power of the equality \eqref{e:main thm} to have well-defined complex functions over $V\subset \mathscr{D}(M_0)$
as explained in \cite{Park17}. The equality conjectured above should be  also compared with main results in \cite{MT}, \cite{MP}.

Now let us explain the structure of this paper. In Section 2, we review some basic facts which are used in the proofs of main results of this paper.
In Section 3, we prove the concerning equality for compact hyperbolic 3-manifolds. In Section 4, we prove Theorem \ref{t:main theorem}.
In Section 5, we prove the corresponding equality to Theorem \ref{t:main theorem} for hyperbolic 3-manifolds with cusps and the representation
$\rho^k$ with odd $k$.

\subsection*{Acknowledgements}
A part of this work was performed while the author visited Research Institute for Mathematical Sciences at Kyoto University. He wants to express gratitude to professor K. Yoshikawa
for his help and hospitality during this period.
This work was partially supported by Samsung Science and Technology Foundation under Project Number SSTF-BA1701-02.

 \section{Basic Materials} 

\subsection{Deformation space of hyperbolic structures}\label{ss:deformation} Suppose that $M_0$ is a complete hyperbolic $3$-manifold of finite volume with $h$ cusps. 
Then $M_0$ has an ideal triangulation \[ M_0= \Delta(z^0_1)\cup \cdots \cup
\Delta(z^0_n). \] Here $\Delta(z^0_i)$ is an ideal tetrahedron described (up to isometry) by the complex number $z^0_i$ 
in the upper half plane such that the Euclidean triangle cut out of any vertex of $\Delta(z^0_i)$ by a
horosphere section is similar to the triangle with vertexes $0$, $1$ and $z^0_i$. If we deform $(z^0_1,\ldots, z^0_n)$ to $(z_1,\ldots,z_n)$ slightly with 
$\mathrm{Im}\, z_i>0$, $i=1,\ldots,n$, then we obtain a complex
$\Delta(z_1)\cup\cdots\cup \Delta(z_n)$ with the same gluing pattern as $M_0$. The necessary and sufficient condition that $\Delta(z_1)\cup\cdots\cup \Delta(z_n)$ 
gives a smooth (not necessarily complete) hyperbolic manifold is that at each edge $e$ of
$\Delta(z_1)\cup\cdots\cup \Delta(z_n)$ the tetrahedron $\Delta(z_i)$ abutting $e$ close up as one goes around $e$, 
and thus the product of the corresponding moduli of $\Delta(z_i)$ at $e$ is $\exp(2\pi i)$ (the product is taken in the universal cover of
${\mathbb{C}}^*$). The consistency condition at $e$ is written as \begin{equation*}\label{e:consis} \prod^n_{i=1} z^{r_i}_i(1-z_i)^{r'_i}=\pm 1 \end{equation*} for some integers $r_i,r'_i$ depending on $e$. Once we have chosen the
numbers $z_i$ satisfying the consistency conditions, $\Delta(z_1)\cup\cdots\cup \Delta(z_n)$ acquires a smooth hyperbolic structure, in general, incomplete. The deformation space $\mathscr{D}(M_0)$ of 
the hyperbolic structures on the underlying topological manifold of $M_0$ is the variety of
$u=(z_1,\ldots,z_n)\in\mathbb{C}^n$ which satisfies the consistency conditions.

Choose a pair of simple closed curves $(m_i,l_i)$ on each torus section $T_i$ of the $i$-th cusp which forms a basis of $H_1(T_i, \mathbb{Z})$. For each 
$z=(z_1,\ldots,z_n)\in \mathscr{D}(M_0)$, let $\rho_z: \pi_1(M_0)\to \mathrm{PSL}(2,\mathbb{C})$ be
a holonomy representation of the corresponding (in)complete hyperbolic manifold $\Delta(z_1)\cup\cdots\cup \Delta(z_n)$. 
We may consider $(m_i,l_i)$ as elements of $\pi_1(M_0)$. If $\rho_z(m_i)$ and $\rho_z(l_i)$ are not parabolic, they have two
fixed points in $\mathbb{C}\cup \{\infty\}$ which we can put at $0$ and $\infty$, so as M\"obius transformations on $\mathbb{C}\cup\{\infty\}$, \[ \rho_u(m_i): w \to a_i w, \qquad \rho_u(l_i): w\to b_i w \]for some $a_i,b_i\in
\mathbb{C}^*$. Set $u_i=\log a_i$ and $v_i=\log b_i$. If $\rho_z(m_i)$ and $\rho_z(l_i)$ are parabolic, we set $u_i=v_i=0$. By \cite{T}, \cite{NZ} we have

\begin{theorem} \label{t:Thurston} {\rm (Thurston \cite{T}, Neumann-Zagier \cite{NZ})} The deformation space $\mathscr{D}(M_0)$ of hyperbolic structures on the underlying topological manifold of $M_0$  
can be holomorphically parametrized by
$(u_1,\ldots,u_h)\in\mathbb{C}^h$ in a neighborhood $V$ of the origin $0=(0,\ldots,0)$ in $\mathscr{D}(M_0)$. For $i=1,\ldots,h$, there are holomorphic functions $\tau_i(u)$ over $V$ such that $v_i= \tau_i(u) u_i$ and $\tau_i(0)$
is in the upper half plane and is the modulus of the Euclidean structure on the torus section $T_i$ associated to the $i$-th cusp of $M_0$ (with respect to $m_i,l_i$). 
\end{theorem}

We denote by $M_u$ the (in)complete hyperbolic 3-manifold corresponding to the point $u=(u_1,\ldots,u_h)\in \mathscr{D}(M_0)$.

By the second statement in Theorem \ref{t:Thurston}, if $u$ is near the origin and $u_i\neq 0$, then $v_i$ is not a real multiple of $u_i$. Hence there is a unique solution $(p_i,q_i)\in\mathbb{R}^2\cup\{\infty\}$ to
\begin{equation}\label{e:prime} p_i u_i+q_i v_i= 2\pi i. \end{equation} We take $(p_i,q_i)=\infty$ if $u_i=0$. This $(p_i,q_i)$ is called the generalized \emph{Dehn surgery coefficient} by Thurston \cite{T}. If each $(p_i,q_i)$ is
a pair of coprime integers, $M_u$ can be completed to a closed hyperbolic manifold denoted by $M_{{p},{q}}$, where ${p}=(p_1,\ldots,p_h)$, $q=(q_1,\ldots, q_h)$, by $(p_i,q_i)$-hyperbolic Dehn surgery to each end of $M_u$.

\subsection{Volume and Chern-Simons invariant}\label{ss:comp-vol}  Over the frame bundle $F(M_u)$ there is a $3$-form $C$ given by \begin{align*} C=&\frac{1}{4\pi^2} \big( 4\theta_1\wedge \theta_2\wedge \theta_3 -d(\theta_1\wedge\theta_{23} +
\theta_2\wedge\theta_{31} +\theta_3\wedge\theta_{12}) \big)\\
  &\ \ +\frac{i}{4\pi^2} \big( \theta_{12}\wedge\theta_{13}\wedge \theta_{23} -\theta_{12}\wedge\theta_1\wedge\theta_2-\theta_{13}\wedge\theta_{1}\wedge\theta_3 -\theta_{23}\wedge\theta_2\wedge\theta_3 \big).
\end{align*} Here $\theta_i$, $\theta_{ij}$ denote the fundamental form and the connection form respectively of the Riemannian connection on $F(M_u)$. Let $s_u$ be the section defined by an orthonormal framing $\mathcal{F}_u$ on a
subset of $M_u$ such that $s_u^* C$ vanishes over $h$ ends of $M_u$. It is called the \emph{simple framing} by Yoshida \cite{Y}. Since $s_u$ satisfies this vanishing condition over the ends, there is an obstruction for $s_u$ to be
defined over whole $M_u$, which is given by a link $L$ inside of $M_u$. Hence, $s_u$ is a section from $M_u\setminus L$ to $F(M_u)$. Let $\kappa_u$ be an orthonormal framing over a tubular neighborhood of $L$ such that its first
component is tangent to $L$ and has the same direction as the first component of $\mathcal{F}_u$ near $L$. For $u\in\mathscr{D}(M_0)$, 
the following complex function is defined by Yoshida \cite{Y}, 
\begin{equation} \label{e:def-f}
f(u) =
\int_{s_u(M_u\setminus L) } C - \frac{1}{2\pi}\int_{s_u(L)} (\theta_1-i\theta_{23}), \end{equation} where $s_u: M_u\setminus L\to F(M_u)$ and $s_u:L\to F(M_u)$ are the sections defined by $\mathcal{F}_u$ and $\kappa_u$
respectively. By the construction, the complex function $f(u)$ defines the complex volume of $M_u$ by $\mathbb{V}(M_u)=2\pi f(u)$ for $u\in\mathscr{D}(M_0)$, in particular, for $M_0$.

The following theorem was conjectured by Nuemann and Zagier \cite{NZ} and was proved by Yoshida \cite{Y}.

\begin{theorem} \label{t:Yoshida} {\rm(Yoshida \cite{Y})} Over a neighborhood $V$ of the origin in  $\mathscr{D}(M_0)$, the complex function $f$ is holomorphic. 
If $u\in V$ represents the hyperbolic manifold $M_u$ which can be completed to
a closed hyperbolic manifold $M_{p,q}$ by $(p_i,q_i)$-hyperbolic Dehn surgery to each end of $M_u$, then 
\begin{align*} \mathrm{Re} f(u) &= \frac{1}{\pi^2} \mathrm{Vol}(M_{p,q}) +\frac{1}{2\pi}\sum_{i=1}^h
\mathrm{length}(\mathbf{g}_i),\\ \mathrm{Im} f(u)&= 2\, CS(M_{p,q})\ +\ \frac{1}{2\pi} \sum_{i=1}^h\mathrm{torsion}(\mathbf{g}_i)  \ \ (\mathrm{mod}\ \mathbb{Z}) \end{align*} 
where $\mathrm{length}(\mathbf{g}_i)$ and $\mathrm{torsion}(\mathbf{g}_i)$ denote the length and the torsion of the closed geodesic $\mathbf{g}_i$ adjoined to the $i$-th end of $M_u$ respectively. \end{theorem}

\subsection{Reidemeister torsion} For an $n$-dimensional vector space over $\mathbb{C}$, let $v=(v_1,\ldots, v_n)$ and $w=(w_1,\ldots, w_n)$ are two bases for it. 
Let $[w/v]$ denote the determinant of the matrix $T$ representing
the change of base from $v$ to $w$, that is, $w_i=\sum t_{ij} v_j$. Suppose 
\begin{equation*} C: C_N\ \stackrel{\partial}{\rightarrow}\ C_{N-1}\ {\rightarrow} \ \cdots \ {\rightarrow}\  C_1\ {\rightarrow}\ C_0 \end{equation*} is a
chain complex of finite complex modules. Let $Z_q$ denote the kernel of $\partial$ in $C_q$, $B_q\subset C_q$ the image of 
$C_{q+1}$ under $\partial$, and $H_q(C)=Z_q/B_q$ the $q$-th homology group of $C$. Choose a base $b_q$ for
$B_q$ for each $q$, and let $\tilde{b}_{q-1}$ be an independent set in $C_q$ such that $\partial \tilde{b}_{q-1}= b_{q-1}$, and $\tilde{h}_q$ be an independent 
set in $Z_q$ representing a base $h_q$ of $H_q(C)$. Then $(b_q,
\tilde{h}_q, \tilde{b}_{q-1})$ is a base for $C_q$. For a given preferred base $c_q$ for $C_q$, note that
$[b_q,\tilde{h}_q, \tilde{b}_{q-1}/c_q]$ depends only on $b_q$, $h_q$, $b_{q-1}$, hence we denote it by $[b_q,{h}_q,
{b}_{q-1}/c_q]$. The \emph{torsion} $\mathcal{T}(C)$ of the chain complex $C$ is the nonzero complex number defined by 
\begin{equation*} \mathcal{T}(C)= \prod_{q=0}^N [b_q,{h}_q, {b}_{q-1}/c_q]^{(-1)^q}. \end{equation*} Note that $\mathcal{T}(C)$
depends only on the choice of the bases $c_q$, $h_q$, but not on the choice of the bases $b_q$.

Let $K$ be a finite cell complex and $\tilde{K}$ the simply connected covering space of $K$ with the fundamental group $\pi_1$ of $K$ 
acting as deck transformations on $\tilde{K}$. Regarding that $\tilde{K}$ is a just the set of
translates of a fundamental domain under $\pi_1$, the chain complex groups $C_q(\tilde{K})$ become modules over the complex group algebra $\mathbb{C}(\pi_1)$ 
with a preferred base consisting of the cells of $K$. Relative to these
preferred base, the boundary operator on the left $\mathbb{C}(\pi_1)$-module $C_q(\tilde{K})$ is a matrix with coefficients in $\mathbb{C}(\pi_1)$. 
For a representation $\chi$ of $\pi_1(K)$ into $\mathrm{SL}(N,\mathbb{C})$, define the chain complex
$C(K,\chi)$ by 
\begin{equation*} C_q(K,\chi) = \mathbb{C}^N\otimes _{\mathbb{C}(\pi_1)} C_q(\tilde{K}) \end{equation*} 
where $\mathbb{C}^N$ is considered as right $\mathbb{C}(\pi_1)$-module via the action of $\chi$. We choose a
preferred base $x_i\otimes e_j$ where $x_i$ runs through a base for $\mathbb{C}^N$ and $e_j$ through the preferred base of 
$C_q(\tilde{K})$ consisting of cells of $K$. Now the \emph{Reidemeister torsion} $\mathcal{T}(K,\chi)$ attached to
the representation $\chi$ is defined by 
\begin{equation*} \mathcal{T}(K,\chi)=\mathcal{T}(C(K,\chi)). \end{equation*} 
A different choice of the preferred bases $e_j$ can give at most sign change of $\mathcal{T}(K,\chi)$ since $\chi$ is a
representation into $\mathrm{SL}(N,\mathbb{C})$. A different choice of the base $x'_i$ for $\mathbb{C}^N$ can 
also give the change by the factor $[x'/x]^{\chi(C)}$ where $\chi(C)$ denotes the Euler characteristic of $C$. Hence, if
$\chi(C)=0$, the Reidemeister torsion $\mathcal{T}(C(K,\chi))$ 
is well-defined as an invariant with a value in $\mathbb{C}^*/\{\pm 1\}$ depending only on the choice of the bases $h_q$ for $H_q(C)$. By \cite{Mi}, it is known that
$\mathcal{T}(C(K,\chi))$ is a combinatorial invariant of $(K, \chi)$. Hence, if $X$ is a compact oriented manifold, any smooth triangulation of 
$X$ gives the same Reidemeister torsion. We denote it by $\mathcal{T}(X,\chi)$.

As before let $M_u$ be a hyperbolic $3$-manifold of finite volume for $u\in \mathscr{D}(M_0)$. 
The holonomy representation $\rho_u:\pi_1(M_u)\to \mathrm{PSL}(2,\mathbb{C})$ can be lifted to $\mathrm{SL}(2,\mathbb{C})$, and the set of such lifts
is in canonical bijection with the set of spin structures on $M_u$. 
For any $k\in\mathbb{N}$, there exists one $k$-dimensional complex irreducible representation $V_k$ of $\mathrm{SL}(2,\mathbb{C})$, which is isomorphic to
$\mathrm{Sym}^{k-1} V_2$, the $(k-1)$-th symmetric tensor of the standard representation $V_2\cong \mathbb{C}^2$. 
Let us compose a lift of $\rho_u$ with $\mathrm{Sym}^{k} V_2$, which is denoted by $\rho_u^{k}$. We can consider
the Reidemeister torsion $\mathcal{T}(M_u, \rho^k_u)$ attached to the representation $\rho^k_u$. In general, this may depend on the choice of bases for $H_*(M_u, \rho^k_u)$.

Let us remark on the convention of notation for $\rho^k_u$ of this paper. We follow the convention of \cite{Park17} where $k$ in $\rho^k_u$ denotes the $k$-th symmetric tensor $\mathrm{Sym}^{k} V_2$.
Hence, this convention is different from some other literatures, for instance, \cite{MFP} where $k$ in $\rho^k_u$ denotes the dimension of $\mathrm{Sym}^{k-1} V_2$.

\subsection{Acyclic spin structure}\label{ss:spin}
A spin structure on $M_0$ naturally induces a spin structure on $M_u$ for $u\in\mathscr{D}(M_0)$. But, a spin structure on $M_0$ can be extended to a spin structure on $M_{p,q}$
with $p=(p_1,\hdots,p_h)$, $q=(q_1,\hdots,q_h)$ obtained by $(p_i,q_i)$-hyperbolic Dehn surgery to each end of $M_u$ only under a condition.
By Proposition 5.2 in \cite{MFP},
a necessary and sufficient condition for this is that \begin{equation}\label{e:spin-cond} \varepsilon_{m_i}^{p_i} \varepsilon_{l_i}^{q_i}=-1 \qquad \text{for \ $i=1,\hdots,h$.} \end{equation} Here $\varepsilon_{m_i}$,
$\varepsilon_{l_i}$ denote the sign of the trace of a $\mathrm{SL}(2,\mathbb{C})$-lifting of $\rho_u(m_i)$, $\rho_u(l_i)$ respectively.
A spin structure on $M_0$ is called \emph{compactly approximable} if there are
infinitely many $p=(p_1,\ldots,p_h)$, $q=(q_1,\hdots,q_h)$ satisfying the conditions \eqref{e:prime} and \eqref{e:spin-cond}. In other words, a spin structure on
$M_0$ is compactly approximable
if there is a sequence $\{M_{p,q}\}$ of infinitely many spin closed hyperbolic manifolds such that the spin structure on $M_{p,q}$ is induced 
from the one of $M_0$.

When $k=2n$,  the representation $\rho_u^{2n}$ 
does not depend on the choice of a spin structure on $M_u$ since the representation
$V_{2n+1}$ factors through $\mathrm{PSL}(2,\mathbb{C})$. But, the homology groups $H_*(M_u,\rho^{2n}_u)$ need not vanish in general. The Reidemeister torsion 
$\mathcal{T}(M_u,\rho_u^{2n})$ is an invariant of $(M_u,\rho_u^{2n})$ valued in $\mathbb{C}^*/\{\pm 1\}$ depending on the choice of
bases of $H_q(M_u,\rho_u^{2n})$.
By Proposition 5.10 in \cite{MFP},  a collection  $\{c_i\}$ of cycles in $H_1(T_i,\mathbb{Z})$ induces a basis of $H_*(M_u, \rho_u^{2n})$.
Hence, the Reidemeister torsion $\mathcal{T}(M_u,\rho_u^{2n})$ can be considered as an invariant of $(M_u,\rho_u^{2n},\{c_i\})$ which we denote by $\mathcal{T}(M_u,\rho_u^{2n},
\{c_i\})$. 

When $k=2n-1$, by Corollary 5.3 in \cite{MFP}, a spin structure on $M_0$ is compactly approximable if and only if it is \emph{acyclic}, that is, 
$H_*(M_0,\rho^{2n-1}_0)=0$ for all $n\in\mathbb{N}$ where $\rho^{2n-1}_0$ is defined by the chosen spin structure. 
Moreover, by the upper semicontinuous property of the
dimension of $H_*(M_u,\rho^{2n-1}_u)$ (see the section 3 of \cite{MFP}), there exists an open neighborhood $V$ of the origin in $\mathscr{D}(M_0)$ 
such that $H_*(M_u,\rho^{2n-1}_u)=0$ for $u\in V$. Hence, the Reidemeister torsion
$\mathcal{T}(M_u,\rho^{2n-1}_u)$ is well defined invariant valued in $\mathbb{C}^*/\{\pm 1\}$ for $u\in V$ if a spin structure over $M_0$ is acyclic. 

\subsection{Ruelle zeta function and Zograf infinite product} 
Let $M$ be a hyperbolic $3$-manifold with finite volume, that is, $M$ may be a compact hyperbolic $3$-manifold, or a noncompact hyperbolic $3$-manifold with cusps. 
For the holonomy representation $\rho:\pi_1(M)\to \mathrm{PSL}(2,\mathbb{C})$, the hyperbolic $3$-manifold $M$ can be realized by the quotient $\Gamma\backslash \mathbf{H}^3$ where
$\Gamma:=\rho(\pi_1(M))\subset \mathrm{PSL}(2,\mathbb{C})$. 
For such a discrete subgroup $\Gamma$, the \emph{critical exponent} $\delta(\Gamma)$ is defined by
\begin{equation}
\delta(\Gamma)= \mathrm{inf} \left\{ s\ \big{\vert} \ \sum_{\gamma\in\Gamma} e^{-s \ell_\gamma} < \infty \right\}.
\end{equation}
It is known that $\delta(\Gamma)<2$ for a discrete group $\Gamma\subset \mathrm{PSL}(2,\mathbb{C})$ as above. 

The \emph{Ruelle zeta function} attached to $\rho^{k}$ is defined by 
\begin{equation*} 
R_{\rho^{k}}(s) = \prod_{[\gamma]} \mathrm{det}\big(\mathrm{Id}
- D_\gamma^{k}\ e^{-s\, l_{\gamma}} \big) \qquad \ \text{for \ \ $\mathrm{Re} (s) > 2+\frac{k}{2}$}. 
\end{equation*} 
Here the product is taken over the set of conjugacy classes of the primitive loxodromic elements $\gamma$ in $\Gamma$,
which can be conjugated to $D_\gamma:=\left(\begin{smallmatrix} e^{\frac12(l_\gamma+i\theta_\gamma)} & 0 \\ 0 &  e^{-\frac12(l_\gamma+i\theta_\gamma)} \end{smallmatrix}\right)$ in $\mathrm{PSL}(2,\mathbb{C})$.
Note that a spin structure on
$M$ is used to lift $D_\gamma$ to $\mathrm{SL}(2,\mathbb{C})$ when $k=2n-1$. 
The domain of convergence $\{s\in \mathbb{C}\, |\, \mathrm{Re}(s) > 2+\frac{k}{2}\}$ for $R_{\rho^k}(s)$ follows from the fact $\delta(\Gamma)<2$.
From the definition of $D_\gamma$, 
we have the following equalities 
\begin{equation}\label{e:decomp} \begin{split} 
&R_{\rho^{2n}}(s)= \prod_{-n\leq m \leq n} R(\sigma_{2m}, s-m),\\ 
&R_{\rho^{2n-1}}(s)=\prod_{-n\leq m \leq n-1} R(\sigma_{2m+1}, s-m-\frac12), \end{split} 
\end{equation} 
where $R(\sigma_k,s)= \prod_{[\gamma]} ( 1- e^{\frac{k}2i\theta_\gamma} e^{-s\, l_\gamma})$ is defined for
$\mathrm{Re}(s)>2$. 

Now let us introduce the Zograf infinite products for a hyperbolic 3-manifold $M$. These are defined by 
\begin{equation*}
\begin{split}
F_n(M,s) = \prod_{m=n}^\infty&
R(\sigma_{-2m}, s+m) \qquad  \ \ \text{for \ \ $\mathrm{Re} (s) > 2-n$, \quad $n\in\mathbb{N}$},\\
G_n(M,s)= \prod_{m=n}^\infty& R(\sigma_{-(2m+1)}, s+m+\frac12) \qquad \ \ \text{for \ \ $\mathrm{Re} (s) > \frac 32-n$, \quad $n\in\mathbb{N}\cup \{0\}$}.
\end{split}
\end{equation*} 
Note that the definition of $G_n(M,s)$ involves a choice of spin structure on $M$.
Evaluating $F_n(M,s)$, $G_n(M,s)$ at $s=0$, we also define
\begin{equation*} 
\begin{split} 
F_n(M):=F_{n}(M,0)=&
\prod_{m=n}^\infty R(\sigma_{-2m},m)=\prod_{[\gamma]}\prod_{m=n}^\infty  (1 - \q_\gamma^{m}) \qquad \text{for $n\geq 3$},\\ 
G_{n}(M):=G_n(M,0)=& \prod_{m=n}^\infty R(\sigma_{-(2m+1)},m+\frac12)=\prod_{[\gamma]}\prod_{m=n}^\infty  (1 -
\q_\gamma^{m+\frac12}) \qquad \text{for $n\geq 2$}. 
\end{split} 
\end{equation*} 
Here $q_\gamma=\exp(-(l_\gamma+\theta_\gamma))$ where $l_\gamma$ and $\theta_\gamma$ denote the length and the torsion of the prime geodesic determined
by the conjugacy class of $\gamma\in \Gamma$. A function of this type was first introduced by Zograf in \cite{Z}, and played
crucial roles in \cite{MT}, \cite{MP}.

By the same way as the case of $M_0$, one can consider the corresponding objects for $M_u$ for $u$ in a small open neighborhood $V$ of the origin in $\mathscr{D}(M_0)$
since the convergence condition is an open condition for $u$. Hence, for $u\in V$ we define
\begin{equation}\label{e:Def-Zog} 
\begin{split} 
F_{n}(M_u):=&\prod_{[\gamma]}\prod_{m=n}^\infty  (1 -
\q_{\gamma}^{m}) \qquad \text{for $n\geq 3$} ,\\ 
G_{n}(M_u):=&\prod_{[\gamma]}\prod_{m=n}^\infty  (1 - \q_{\gamma}^{m+\frac12})  \qquad \text{for $n\geq 2$}. \end{split} \end{equation} 
Here the first product is taken over the set of
conjugacy classes of the primitive loxodromic elements $\gamma$ defined by $\rho_u$ which are deformations of the loxodromic elements defined by $\rho_0$, and $\q_{\gamma}$ is defined as before using the holonomy representation
$\rho_u$ for $u$ in an open neighborhood $V$ of the origin in $\mathscr{D}(M_0)$. 
Let us remark that there are loxodromic elements defined by $\rho_u$ which are not deformations of loxodromic elements defined by $\rho_0$.
Note that $F_{n}(M_u)$ and $G_{n}(M_u)$ are holomorphic functions over $V$ since we do not take such elements in the definition \eqref{e:Def-Zog}.

\section{Case of compact hyperbolic 3-manifolds}

In this section, we prove an equality involving moduli of Reidemeister torsion, complex volume, and Zograf infinite product for compact hyperbolic 3-manifolds.
This is one of main ingredients in the proof of the same type equality for noncompact hyperbolic 3-manifolds with cusps.
A more refined equality between complex valued invariants was obtained in \cite{Park17} for the case of compact hyperbolic manifolds. 
But, this equality contains some additional terms of modulus 1 which are not easy to treat in the proofs given in following sections. See Remark \ref{r:zero-torsion}.
Hence, we use the following theorem, which holds only with modulus signs, to derive main results of this paper.

\begin{theorem}\label{t:closed-case}
For a closed hyperbolic manifold $M$, the following equalities hold
\begin{align}
\left| \mathcal{T} (M,\rho^{2(n-1)})\right|^{-1} =& \left| \exp \left(\frac{1}{\pi}( n^2-n+\frac16) \mathbb{V}(M)\right) F_n(M) \right| \qquad \text{for}\  \  n\geq 3, \label{e:main-result-odd}\\
\left| \mathcal{T} (M,\rho^{2n-1})\right|^{-1} =& \left| \exp \left(\frac{1}{\pi}( n^2-\frac1{12}) \mathbb{V}(M)\right) G_n(M) \right| \qquad \text{for}\  \  n\geq 2. \label{e:main-result-even}
\end{align}
where $\mathbb{V}(M)= \left(\mathrm{Vol}+i2\pi^2 \mathrm{CS}\right)(M)$. 
\end{theorem}

Note that the homology groups $H_*(M,\rho^k)$ vanish under conditions in Theorem \ref{t:closed-case}.
 
\begin{proof}
The proof is simpler than the one given in \cite{Park17} since we do not care of phase parts.

First, we denote by $E(\rho^k)$ the flat vector bundle over $M$ defined by $\rho^k$. Then there exists a canonical Hermitian metric over each fiber of $E(\rho^k)$
constructed in \cite{MM}, which is called as \emph{admissible metric}. Using the hyperbolic metric over $M$ and this admissible metric for $E(\rho^k)$, one can define the Laplaicans $\Delta_p$ 
acting on $\Omega^p(M, E(\rho^k))$ for $p=0,1,2,3$,
which are selfadjoint nonnegative operators. Then, in \cite{Wot} the following equality was proved
\begin{equation}\label{e:pf-real1}
|R_{\rho^k}(0)|= T(M,\rho^k) 
\end{equation}
where $T(M,\rho^k)$ denotes the analytic torsion defined by the Laplacians $\Delta_p$ acting on $\Omega^p(M, E(\rho^k))$ for $p=0,1,2,3$.
By \cite{BZ},\cite{Mu}, we also have
\begin{equation}\label{e:pf-real2}
T(M, \rho^k)= |\mathcal{T}(M,\rho^k)|^2. 
\end{equation}
By \eqref{e:decomp} and the
following function equation of $R(\sigma_k,s)$ proved in the chapter 4 of \cite{BO}, 
\begin{equation}\label{e:pf-real3} 
|R(\sigma_k,s)| =\left|\exp\left(\frac{4}{\pi} \mathrm{Vol}(M) s\right) \, R(\sigma_{-k},-s)\right|, 
\end{equation} 
we have
\begin{equation}\label{e:pf-real4} 
\begin{split} 
&|R_{\rho^{2(n-1)}}(0)|=\Big(\prod_{-(n-1)\leq m \leq n-1} |R(\sigma_{2m}, s-m)|\Big)_{s=0} \\
=& \exp\left(- \frac{2n(n-1)}{\pi}\mathrm{Vol}(M)\right) \, \left(|R(\sigma_0,s)| \prod_{1\leq m
\leq n-1} | R(\sigma_{-2m},s+m)|^2\right)_{s=0}. 
\end{split} \end{equation} 
For $n\geq 3$, the last terms on the right hand side of \eqref{e:pf-real4} can be written as follows, 
\begin{equation}\label{e:pf-real5} \begin{split} &
\Big(|R(\sigma_0,s)| \prod_{1\leq m \leq n-1} | R(\sigma_{-2m},s+m)|^2\Big)_{s=0}\\
 =& \Big( |R(\sigma_0,s)|\, |F_1(M,s)|^2 \Big)_{s=0}\, \prod_{m=n}^\infty|R(\sigma_{-2m},m)|^{-2}\\
=& \exp\left(-\frac{1}{3\pi} \mathrm{Vol}(M)\right)\,  \prod_{m=n}^\infty|R(\sigma_{-2m},m)|^{-2}. \end{split} \end{equation}
The second equality above follows from Theorem 3.8 in \cite{Park17} recalling that $R(\sigma_0,s)=R_{\rho^0}(s)$.   By \eqref{e:pf-real1},
\eqref{e:pf-real2} and \eqref{e:pf-real5}, 
\begin{equation}\label{e:pf-real6} 
|\mathcal{T}(M,\rho^{2(n-1)})|^{-1} = \exp\left(  \frac{1}{\pi}(n^2-n+\frac16)\, \mathrm{Vol}(M) \right)\,  \prod_{m=n}^\infty|R(\sigma_{-2m},m)|. \end{equation}
This completes the proof for the case of $\rho^{2(n-1)}$. The case for $\rho^{2n-1}$ can be proved in the same way.
\end{proof}

\begin{remark}\label{r:zero-torsion}
Comparing  equalities \eqref{e:main-result-odd}, \eqref{e:main-result-even} with Theorems 1.1 and 5.1 in \cite{Park17}, one can conclude that
the torsions defined by the chain complex of the zero generalized eigenspaces, 
which are denoted by $\mathcal{T}_0(\mathcal{M}_\Gamma, \rho_{k})$ in \cite{Park17}, are actually of modulus 1.
Hence, one may wonder whether actually this would be equal to $1$, but it seems to be difficult to prove it using the method in \cite{Park17}.
\end{remark}

\section{Case of hyperbolic 3-manifolds with cusps and $\rho^{2(n-1)}$}\label{s:pf-main-thm}

For a given hyperbolic 3-manifold with cusps $M_0$, we take a sequence $\{M_{p,q}\}$ of infinitely many compact hyperbolic 3-manifolds, which are obtained by
$(p_i,q_i)$-hyperbolic Dehn surgery to each end of $M_u$ for points $u$ near the origin in $\mathscr{D}(M_0)$ satisfying \eqref{e:prime}. 

Now, for a closed hyperbolic manifold $M_{p,q}$, by Theorem \ref{t:closed-case}, we have
\begin{equation}\label{e:pf-odd-1}
\left| \mathcal{T} (M_{p,q},\rho^{2(n-1)})\right|^{-1} = \left| \exp \left(\frac{1}{\pi}( n^2-n+\frac16) \mathbb{V}(M_{p,q})\right) F_n(M_{p,q}) \right| .
\end{equation}
Here $\rho^{2(n-1)}$ is the representation of $\pi_1(M_{p,q})$ to $\mathrm{Sym}^{2(n-1)} V_2$, which is defined 
by the same way as  $\rho^{2(n-1)}_u$ for the corresponding point $u\in \mathscr{D}(M_0)$.

By the Mayer-Vietoris argument for the Reidemeister torsion as in Lemma 3.12 of \cite{MFP},
or in the section 3 of \cite{P97}, we have 
\begin{equation}\label{e:pf-odd-2} 
\mathcal{T}(M_{p,q}, \rho^{2(n-1)})=\mathcal{T}(M_u,\rho^{2(n-1)}, \{p_i m_i + q_il_i\}) \prod_{i=1}^h \prod_{j=1}^{n-1} (\q_{\gamma_i}^m -1)(\q_{\gamma_i}^{-m}-1)
\end{equation} 
where $\q_{\gamma_i}=\exp(-(l_{\gamma_i}+i\theta_{\gamma_i}))$. Here $\gamma_i$, $i=1,\ldots, h$ 
denotes the primitive loxodromic element corresponding to the added closed geodesic $\mathbf{g}_i$ to the $i$-th end of
$M_u$. Let us remark the equality \eqref{e:pf-odd-2} holds up to sign since the definition of Reidemeister torsion has $\pm1$ ambiguity.  
The same remark also holds for equalities for Reidemeister torsion in the following parts of proof.

For the complex volume of $M_{p,q}$, by Theorem \ref{t:Yoshida}, we have
\begin{equation}\label{e:pf-odd-3}
\exp\left( \frac{2}{\pi} \mathbb{V}(M_{p,q})\right) = \exp\left( \frac{2}{\pi} \mathbb{V}(M_u)\right) 
\prod^h_{i=1} q_{\gamma_i}.
\end{equation}

For the Zograf infinite product $F_n(M_{p,q})$, we separate the terms of $\gamma_i$ for $i=1,\ldots,h$ from other terms by
\begin{equation}\label{e:pf-odd-4}
F_n(M_{p,q}) = \prod^h_{i=1}\prod^{\infty}_{m=n} (1- q_{\gamma_i}^m)^2  \prod_{[\gamma]\neq [\gamma_i]} \prod^\infty_{m=n} (1-q_\gamma^m). 
\end{equation}
Here note that the primitive conjugacy classes corresponding $\gamma_i$ and $\gamma_i^{-1}$ contribute by the same factor so that we get  $(1- q_{\gamma_i}^m)^2$.
By Theorem 6.5 in \cite{MFP}, we have
\begin{equation}\label{e:pf-odd-conv}
\lim_{u\to 0} \left( \prod_{[\gamma]\neq [\gamma_i]} \prod^\infty_{m=n} (1-q_\gamma^m) \right) = F_n(M_0)
\end{equation}
where $u\to 0$ means that $(p_i,q_i)$ changes as $u=(u_1,\ldots, u_h)$ goes to the origin of $\mathscr{D}(M_0)$ satisfying the condition \eqref{e:prime}.
From now on, the limit as $u\to 0$ should be understood in this sense.

Combining the equalities \eqref{e:pf-odd-1}, \eqref{e:pf-odd-2}, \eqref{e:pf-odd-3}, and \eqref{e:pf-odd-4}, we have
\begin{equation}\label{e:pf-odd-5}
\begin{split}
&\left| \mathcal{T} (M_{u},\rho^{2(n-1)}_u, \{p_i m_i + q_il_i\})\, \prod_{i=1}^h q_{\gamma_i}^{\frac1{12}} \prod_{m=1}^{\infty} (1-\q_{\gamma_i}^m)^2\right|^{-1}\\
=&\, \left| \exp \left(\frac{1}{\pi}( n^2-n+\frac16) \mathbb{V}(M_u)\right) \left(F_n(M_0)+\varepsilon_1(u)\right) \right| 
\end{split}
\end{equation}
where $\varepsilon_1(u)\in\mathbb{C}$  such that $\lim_{u\to 0}\varepsilon_1(u)=0$.

Let us recall the following equality given in the section 4 of \cite{NZ}, 
\begin{equation*} l_{\gamma_i}+i \theta_{\gamma_i} = - (r_i u_i+s_iv_i) \qquad \ (\mathrm{mod}\ 2\pi i) \end{equation*} where $r_i$, $s_i$ are integers such
that $p_is_i-q_ir_i=1$. By this and \eqref{e:prime}, 
\begin{equation}\label{e:pm8} \tilde{\tau}_i(u):= \frac{r_i+s_i \tau_i(u)}{p_i+q_i\tau_i(u)} =-\frac{1}{2\pi i} (l_{\gamma_i}+i \theta_{\gamma_i}) \qquad \ (\mathrm{mod}\
\mathbb{Z}) \end{equation} where $\tau_i(u)$ is given in Theorem \ref{t:Thurston}. 
Let us remark that 
\begin{equation}\label{e:action} \text{ the action of $\big(\begin{smallmatrix} a & b \\ c & d\end{smallmatrix}\big)$ on $\big(\begin{smallmatrix} m_i\\l_i\end{smallmatrix}\big)$ induces the action of 
$\big(\begin{smallmatrix} d & c \\ b & a \end{smallmatrix}\big)$ on $\big(\begin{smallmatrix}\tau\\1\end{smallmatrix}\big)$} \end{equation} 
as we obtained in \eqref{e:pm8}. 
Note that there exists an open neighborhood $V$ of the origin of $\mathscr{D}(M_0)$ such that $\tilde{\tau}_i(u)$ lies in the upper half plane for $u\in V$. 
Hence we can consider the Dedekind eta function of
$\tilde{\tau}_i(u)$ for $u\in V$. 
Recall that the Dedekind eta function $\eta(\tau)$ for $\tau$ in the upper half plane is defined by 
\begin{equation*} \eta(\tau)= e^{\frac{2\pi i \tau}{24}} \prod_{m=1}^\infty \big( 1- \exp(2 \pi
im\tau)\big), \end{equation*} which satisfies the following transformation law, 
\begin{equation}\label{e:pf-odd-7} 
\log \eta \big(\frac{d\tau+c}{b\tau+a}\big)=\log \eta(\tau) +\frac{1}{4} \log(-(a+b\tau)^2) +\frac{1}{12} \, \pi i
I \end{equation} where $I$ is an integer depending on $\big(\begin{smallmatrix} d & c \\ b & a \end{smallmatrix}\big)$.

By the definition of the Reidemeister torsion, we have the following equality
\begin{equation}\label{e:pf-odd-6}
\mathcal{T}(M_u,\rho^{2(n-1)}_u, \{p_i m_i + q_il_i\})=\mathcal{T}(M_u,\rho^{2(n-1)}_u, \{m_i\})\, A_{2(n-1)}(u)^{-1}, 
\end{equation} 
where $A_{2(n-1)}(u)$ denotes the determinant of the basis changing matrix from the one determined by $\{m_i\}$ to the one determined by $\{p_im_i+q_il_i\}$ as explained in the subsection \ref{ss:spin}. 
Hence, using \eqref{e:pf-odd-6} the equality \eqref{e:pf-odd-5} can be re-written in terms of the Dedekind eta function $\eta(\tilde{\tau}_i(u))$ as follows,
\begin{equation}\label{e:pf-odd-8} 
\begin{split} 
&\left| \mathcal{T}(M_u,\rho^{2(n-1)}_u, \{m_i\}) \, A_{2(n-1)}(u)^{-1} \, \prod_{i=1}^h \eta(\tilde{\tau}_i(u))^2 \right|^{-1} \\ 
&\qquad \quad =\left| \exp\left(\frac{1}{\pi} (n^2-n+\frac16) \, \mathbb{V}(M_u) \right)\, \left(F_{n}(M_0)+\varepsilon_1(u)\right)\right| .
\end{split} 
\end{equation} 
As in the proof of Lemma 5.13 of \cite{MFP}, one can check that
\begin{equation}\label{e:pf-odd-A}
\lim_{u\to 0} \left( A_{2(n-1)}(u)^{-1} \, \prod^h_{i=1} (p_i+q_i\tau_i(u)) \right)=1.
\end{equation}
This and the equality \eqref{e:pf-odd-7} imply
\begin{equation}\label{e:est-quot}
 \lim_{u\to 0} \left( A_{2(n-1)}(u)^{-1}\, \prod^h_{i=1} \eta(\tilde{\tau}_i(u))^2 \right) =  \prod^h_{i=1} \eta(\tau_i(0))^2.
\end{equation}
Hence, we have
\begin{equation}\label{e:pf-odd-10} 
\begin{split} 
&\left| \mathcal{T}(M_u,\rho^{2(n-1)}_u, \{m_i\}) \, \left(\prod_{i=1}^h \eta({\tau}_i(0))^2+\varepsilon_2(u) \right)  \right|^{-1} \\ 
&\qquad \quad =\left| \exp\left(\frac{1}{\pi} (n^2-n+\frac16) \, \mathbb{V}(M_u) \right)\, \left(F_{n}(M_0)+\varepsilon_1(u)\right)\right| .
\end{split} 
\end{equation} 
where $\varepsilon_2(u)\in\mathbb{C}$  such that $\lim_{u\to 0}\varepsilon_2(u)=0$.
Taking $u\to 0$ along the discrete set corresponding to the sequence $\{M_{p,q}\}$, we obtain the corresponding equality for $M_0$. This completes the proof of Theorem \ref{t:main theorem}.

\section{Case of hyperbolic 3-manifolds with cusps and $\rho^{2n-1}$}\label{s:even}

In this section we prove

\begin{theorem}\label{t:theorem-even} Let $M_0$ be a complete hyperbolic $3$-manifold of finite volume with $h$ cusps. For an acyclic spin structure on $M_0$,  
the following equality holds for $n\geq 2$, 
\begin{equation}\label{e:thm-even} 
\left| \mathcal{T} (M_0,\rho^{2n-1})\, \prod_{i=1}^h \left(\, \theta_{01}(0,\tau_i) \eta({\tau}_i)^{-1}\, \right) \right|^{-1} \\
=\left| \exp\left( \frac{1}{\pi} (n^2-\frac{1}{12}) \mathbb{V}(M_0) \right)\, G_{n}(M_0) \right|. 
\end{equation} 
Here $\theta_{01}(z,\tau)$ is a theta function defined by 
\begin{equation} \theta_{01}(z,\tau)=\sum_{n\in\mathbb{Z}} \exp\big(\pi i n^2\tau +2\pi in (z+\frac12)\big)
\end{equation}
 for $z\in\mathbb{C}$ and $\tau$ in the upper half plane.
\end{theorem}

\begin{proof}
Although we prove this theorem essentially in the same way as the proof of Theorem \ref{t:main theorem}, we need to explain how the acyclic spin structure  is involved in the following proof.
If the given spin structure on $M_0$ is  acyclic, then it is compactly approximable as explained in the subsection \ref{ss:spin}. Then, for a basis $(m_i,l_i)$ of $H_1(T_i,\mathbb{Z})$
for $i=1,\ldots,h$, there are infinitely many points near the origin of $\mathscr{D}(M_0)$ with coprime pairs
$(p_i,q_i)$ $i=1,\ldots,h$ satisfying \eqref{e:prime} and \eqref{e:spin-cond}. In particular, if needed properly changing the basis $(m_i,l_i)$ to have $\varepsilon_{m_i}=-1$ for $i=1,\ldots,h$, we may assume that 
\begin{equation}\label{e:assump}
p_i=4k_i+1, \quad  q_i=4l_i \qquad \text{for} \quad k_i,l_i\in \mathbb{Z}.
\end{equation}
This assumption will play a crucial role later.  
For the closed hyperbolic manifold $M_{p,q}$ with the induced spin structure, by Theorem \ref{t:closed-case} we have
\begin{equation}\label{e:pf-even-1}
\left| \mathcal{T} (M_{p,q},\rho^{2n-1})\right|^{-1} = \left| \exp \left(\frac{1}{\pi}( n^2-\frac1{12}) \mathbb{V}(M_{p,q})\right) G_n(M_{p,q}) \right| .
\end{equation}

By the Mayer-Vietoris argument for the Reidemeister torsion as in Lemma 3.7 of \cite{MFP},
we have 
\begin{equation}\label{e:pf-even-2} 
\mathcal{T}(M_{p,q}, \rho^{2n-1})= \mathcal{T} (M_u,\rho^{2n-1}) \prod_{i=1}^h \prod_{m=0}^{n-1} (\q_{\gamma_i}^{m+\frac12} -1)(\q_{\gamma_i}^{-(m+\frac12)}-1) \end{equation} where
$\q_{\gamma_i}^{j+\frac12}=\exp(-(j+\frac12)(l_{\gamma_i}+i\theta_{\gamma_i}))$ is defined with respect to the induced spin structure on $M_{p,q}$.

For the Zograf infinite product $G_n(M_{p,q})$, by Theorem 6.5 in \cite{MFP},
\begin{equation}\label{e:pf-even-3}
G_n(M_{p,q}) = \prod^h_{i=1}\prod^{\infty}_{m=n} (1- q_{\gamma_i}^{m+\frac12})^2 \left(G_n(M_0)+ \varepsilon_3(u)\right) 
\end{equation}
where $\varepsilon_3(u)\in\mathbb{C}$ such that $\lim_{u\to 0}\varepsilon_3(u)=0$.
Here note that the primitive conjugacy classes corresponding $\gamma_i$ and $\gamma_i^{-1}$ contribute by the same factor so that we get  $(1- q_{\gamma_i}^{m+\frac12})^2$.

Combining equalities \eqref{e:pf-odd-3}, \eqref{e:pf-even-1}, \eqref{e:pf-even-2}, and \eqref{e:pf-even-3},
\begin{equation}\label{e:pf-even-4}
\begin{split}
&\left| \mathcal{T}(M_u,\rho^{2n-1}_u) \prod^h_{i=1} q_{\gamma_i}^{-\frac{1}{24}} \prod^{\infty}_{m=0} (1- q_{\gamma_i}^{m+\frac12})^2 \right|^{-1}\\
=& \left|  \exp\left(\frac{1}{\pi} (n^2-\frac{1}{12}) \, \mathbb{V}(M_u) \right)\, \left(G_n(M_0)+ \varepsilon_3(u)\right) \right|.
\end{split}
\end{equation}

From the formula given at p. 69 in \cite{Mumford}, let us recall that the theta function $\theta_{01}(z,\tau)$ has a product expression at $z=0$, 
\begin{equation} \theta_{01}(0,\tau)= \prod_{m=1}^\infty \big( 1- \exp(2\pi i m\tau) \big)\, \prod_{m=0}^\infty \big(1- \exp(\pi i
(2m+1)\tau)\big)^2. \end{equation} 
By Proposition 9.2 in \cite{Mumford}, it also satisfies the transformation law 
\begin{equation}\label{e:pf-even-5} 
\theta_{01}\big(0, \gamma\tau)^2 =  (b\tau+a) \theta_{01} (0,\tau)^2
\end{equation} 
for $\gamma\in \Gamma(4)$. Here $\Gamma(4)\subset \mathrm{SL}(2,\mathbb{Z})$ denotes the principal congruence group of level $4$, and
the action $\gamma\tau$ is
given by $\frac{d\tau+c}{b\tau+a}$ for an element 
$\gamma=\big(\begin{smallmatrix} d & c \\ b & a \end{smallmatrix}\big)$ in $\Gamma(4)$ recalling \eqref{e:action}. By the assumption given in \eqref{e:assump}, we
have
\begin{equation}
\gamma_i=\begin{pmatrix} s_i & r_i \\ q_i & p_i \end{pmatrix} \in \Gamma(4). 
\end{equation}
Hence, for these $\gamma_i$, $i=1,\ldots, h$,  we have the following equality, 
\begin{equation}\label{e:pf-even-6} 
\begin{split} 
&|\q_{\gamma_i}|
\prod_{m=0}^\infty |(1- \q_{\gamma_i}^{m+\frac12})|^{-48}\\ 
=& |\exp (2 \pi i \tilde{\tau}_i(u))|\, \prod_{m=0}^\infty |(1- \exp(\pi i (2m+1)\tilde{\tau}_i(u))|^{-48}\\ =&| \eta(\tilde{\tau}_i(u))|^{24}\, |
\theta_{01}(0,\tilde{\tau}_i(u))|^{-24} = |\eta({\tau}_i(u))|^{24}\, |\theta_{01}(0,{\tau}_i(u))|^{-24}. 
\end{split} 
\end{equation} 
By \eqref{e:pf-even-4} and \eqref{e:pf-even-6}, we have 
\begin{equation}\label{e:pf-even-7} 
\begin{split}
&\left|\mathcal{T}(M_u,\rho^{2n-1}_u) \,\prod_{i=1}^h \left(\,  \theta_{01}(0,{\tau}_i(u))\, \eta({\tau}_i(u))^{-1}\, \right)\right|^{-1}\\ 
=&\left| \exp\left(\frac{1}{\pi} (n^2-\frac{1}{12}) \, \mathbb{V}(M_u) \right)\, (G_{n}(M_0)+\varepsilon_3(u) ) \right|. 
\end{split} 
\end{equation} 
Taking $u\to 0$ along the discrete set corresponding to the sequence $\{M_{p,q}\}$, we obtain the corresponding equality for $M_0$ with acyclic spin structure. This completes the proof for
the case of hyperbolic 3-manifolds with cusps and $\rho^{2n-1}$.
\end{proof}

By the same reasoning as before, Theorem \ref{t:theorem-even} leads the author to make the following conjecture:

\begin{conj}
 There exists an open neighborhood $V$ of the origin in $\mathscr{D}(M_0)$ where the following equality holds for $n\geq 2$,
\begin{equation*}\label{e:conjecture-even} 
\begin{split}
&\mathcal{T}(M_u,\rho_u^{2n-1} )^{-12} \, \prod_{i=1}^h  \left(\, \theta_{01}(0,\tau_i(u)) \eta({\tau}_i(u))^{-1}\, \right)^{-12}\\
=&\ c_{M_0,n} \exp\left( \frac{1}{\pi} (12n^2-1) \mathbb{V}(M_u) \right)\, G_{n}(M_u)^{12} 
\end{split}
 \end{equation*}
where $c_{M_0,n}$ is a constant depending only on $M_0$ and $n$ with $|c_{M_0,n}|=1$.
\end{conj} 
Let us remark that we need to take 12-th power of the equality \eqref{e:main thm} to have well-defined complex functions over $V\subset \mathscr{D}(M_0)$
as explained in \cite{Park17}. 


\bibliographystyle{plain}

\end{document}